\documentclass[10pt,reqno,a4paper]{amsart}
\usepackage{bm}
\usepackage{mathrsfs}
\usepackage{amsmath}
\usepackage{amsfonts}
\usepackage{mathtools}
\usepackage{amssymb}
\usepackage{xcolor}
\usepackage{graphicx}
\usepackage{hyperref}
\usepackage{cleveref}
\usepackage{float}
\usepackage{enumitem}
\setlist[enumerate,1]{label=\textup{(\arabic*)}}

\hypersetup{
	colorlinks=true,
	linkcolor=blue!50!black,      
	citecolor=green!70!black,
    urlcolor=cyan!50!blue,
}

\newtheorem{theorem}{Theorem}[section]
\newtheorem{corollary}[theorem]{Corollary}
\newtheorem{lemma}[theorem]{Lemma}
\newtheorem{proposition}[theorem]{Proposition}

\newtheorem{question}[theorem]{Question}

\newtheorem{remark}[theorem]{Remark}

\newtheorem{example}[theorem]{Example}
\numberwithin{equation}{section}

\def\blt{\bullet}

\def\O{\mathscr{O}}
\def\J{\mathscr{J}}
\def\I{\mathscr{I}}

\def\a{\mathfrak{a}}
\def\b{\mathfrak{b}}
\def\m{\mathfrak{m}}
\def\q{\mathfrak{q}}

\def\ol{\overline}
\def\wt{\widetilde}
\def\wh{\widehat}

\def\fk{\mathfrak}

\def\bbQ{\mathbb{Q}}
\def\bbZ{\mathbb{Z}}
\def\bbR{\mathbb{R}}
\def\bbN{\mathbb{N}}
\def\bbC{\mathbb{C}}

\DeclareMathOperator{\ord}{ord}
\DeclareMathOperator{\Spec}{Spec}
\DeclareMathOperator{\lct}{lct}
\DeclareMathOperator{\val}{val}
\DeclareMathOperator{\Val}{Val}
\DeclareMathOperator{\QM}{QM}

\newcommand*{\vb}{\ol{\nu}}
\newcommand*{\vbI}{\ol{\nu}_{I_{\blt}}}
\newcommand*{\vI}{\nu_{I_{\blt}}}
\newcommand*{\vbJ}{\ol{\nu}_{J_{\blt}}}
\newcommand*{\vJ}{\nu_{J_{\blt}}}

\begin{document}

\title[Existence of valuative interpolation]{An algebraic approach to the existence of valuative interpolation}

\author{Shijie Bao}
\address{Shijie Bao: Academy of Mathematics
	and Systems Science, Chinese Academy of Sciences, Beijing 100190, China}
\email{bsjie@amss.ac.cn}

\author{Qi'an Guan}
\address{Qi'an Guan: School of Mathematical Sciences,
	Peking University, Beijing, 100871, China}
\email{guanqian@math.pku.edu.cn}

\author{Zhitong Mi}
\address{Zhitong Mi: School of Mathematics and Statistics, Beijing Jiaotong University, Beijing, 100044, China}
\email{zhitongmi@amss.ac.cn}

\author{Zheng Yuan}
\address{Zheng Yuan: State Key Laboratory of Mathematical Sciences, Academy of Mathematics
	and Systems Science, Chinese Academy of Sciences, Beijing 100190, China}
\email{yuanzheng@amss.ac.cn}



\date{\today}

\begin{abstract}
    An algebraic approach is presented for the valuative interpolation problem, which recovers and generalizes prior characterizations known in the complex analytic setting by the authors. We use the asymptotic Samuel function to give the characterization of the existence of valuative interpolation. We also give a characterization of the existence in the infinite valuative interpolation problem.
\end{abstract}


\maketitle

\setcounter{tocdepth}{1}
\tableofcontents
\section{Introduction}

Let $R$ be an excellent regular domain of equicharacteristic $0$ and of dimension $n$.

Let $\a_1,\ldots,\a_r$ be $r$ nonzero ideals of $R$, and let $b_1,\ldots,b_r\in\mathbb{R}_{\ge 0}$. We consider the following valuative interpolation problem which was considered by the authors in the complex analytic setting (e.g. \cite[Theorem 1.2]{BGMY25} deals with the case where $R=\O_{\bbC^n,o}$ the ring of holomorphic germs):

\begin{question}
Can one give a necessary and sufficient condition for the existence of a valuation $v$ on $R$ such that $v(\a_j)=b_j$ for $j=1,\ldots, r$?
\end{question}
    
We introduce some notations for explaining our results in the present paper. Suppose $\a$ is an ideal on $X$, and $\lambda\in\bbR_{\ge 0}\cup\{\infty\}$. Denote by $(\lambda\cdot\a)_{\blt}=\big((\lambda\cdot\a)_m\big)_{m\in\bbN}$ the \emph{filtration} (cf. \Cref{subsec-gr.fil}) of ideals such that each term:
\begin{equation*}
    (\lambda\cdot\a)_m=
    \begin{cases}
        R & \lambda=0 \\
        \a^{\lceil \lambda m\rceil} & 0<\lambda<\infty \\
        0 & \lambda=\infty \\
    \end{cases}
    \quad \quad m\in\bbN.
\end{equation*}

If $(\b_1)_{\blt},\ldots,(\b_s)_{\blt}$ are filtrations of ideals in $R$, then the summation $\fk{c}_{\blt}\coloneqq \sum(\b_i)_{\blt}$ of these filtrations is a filtration defined as:
\begin{equation*}
        \fk{c}_{m}=\sum_{\substack{m_1+\cdots+m_s=m \\ m_i\in\bbN}}(\b_1)_{m_1}\cdots(\b_s)_{m_s}, \quad m\in\bbN.
\end{equation*}

Now we return to the valuative interpolation problem. For the ideals $\a_1,\ldots,\a_r$ and non-negative real numbers $b_1,\ldots,b_r$, if all of $b_j$ are zero, then we can directly choose the trivial valuation on $R$, so we only consider the case where at least one of them is nonzero. Hence, may assume $\sum_{b_j>0}\a_j\neq R$, otherwise none of valuations $v$ on $R$ can satisfy $v(\a_j)=b_j$ for all $j$ (cf. \Cref{subsec-val}).

Consider the filtration
\[I_{\blt}\coloneqq \Big(\frac{1}{b_1}\cdot \a_{1}\Big)_{\blt}+\cdots+\Big(\frac{1}{b_r}\cdot \a_{r}\Big)_{\blt}.\]
Following Cutkosky and Peaharaj \cite[Definition 3.5]{CP24}, the \emph{asymptotic Samuel function} associated with this filtration is defined as follows. For any ideal $\a$ of $R$, let
\[\vbI(\a)\coloneqq \lim_{k\to\infty}\frac{\vI\big(\a^k\big)}{k}\in \bbR_{\ge 0}\cup\{\infty\},\]
where $\vI\big(\a\big)\coloneqq \sup\,\{m\in\bbN \mid \a\subseteq I_m\}$ denotes the \emph{order function} of $I_{\blt}$. This construction generalizes the classical asymptotic Samuel function \cite{Sa52} and serves as an algebraic analogue of the relative type introduced in complex analytic settings \cite{Ras06} (cf. \Cref{sec-c.a.c}).

\begin{theorem}\label{thm-finite.interp}
    The following three statements are equivalent.
    \begin{enumerate}
    \item There exists a valuation $v$ on $R$ such that $v(\a_j)=b_j$ for $j=1,\dots,r$.
    \item There exists a quasi-monomial valuation $v$ on $R$ such that $v(\a_j)=b_j$ for $j=1,\dots,r$.
    \item The equality $\vbI(\a_1\cdots\a_r)=\sum_{j=1}^r b_j$ holds.
    \end{enumerate}
\end{theorem}

In fact, the equivalence of the first two statements, along with the equivalence of the existence of \emph{Zhou valuations} (cf. \cite{BGMY24,BGZ25}) satisfying the same interpolation conditions, has been proved in \cite[Appendix A.1]{BGMY25}.

In addition, for the infinite valuative interpolation problem, we apply the method of valuation approximation to prove the following result in the present paper.

\begin{theorem}\label{thm-infinite.interp}
    Additionally assume $R$ is local with the unique maximal ideal $\m$. Let $(\a_j)_{j\in\bbZ_+}$ be a countable sequence of nonzero ideals in $R$ such that $\sqrt{\sum_{j=1}^{\infty} \a_j}=\m$, and let $(b_j)_{j\in\bbZ_+}$ be a sequence of positive real numbers. Then the following statements are equivalent:
    \begin{enumerate}
        \item There exists a valuation $v$ on $R$ centered at $\m$ with $A(v)<\infty$ such that $v(\a_j)=b_j$ for all $j\ge 1$;
        \item For every $r\ge 1$,
    \[\vb_{I_{\blt}^{(r)}}\big(\a_1\cdots\a_r\big)=\sum_{j=1}^r b_j,\]
    and it holds that
    \[\sup_{r\ge 1}\sup_{k\in\bbN}\Big(\lct^{\a_1^k\cdots\a_r^k}\big( I^{(r)}_{\blt}\big)-k\sum_{j=1}^r b_j\Big)<\infty.\]
    \end{enumerate}
    Here $I^{(r)}_{\blt}=\sum_{j=1}^{r}\big(\frac{1}{b_j}\cdot \a_j\big)_{\blt}$.
\end{theorem}

\vspace{.1in} {\bf Organization}. This paper is organized as follows. In \Cref{sec-r.o}, we recall the notion of asymptotic Samuel function and quickly prove the implication (1) $\Rightarrow$ (3) in \Cref{thm-finite.interp}. In \Cref{sec-f.v.i}, we prove \Cref{thm-finite.interp} by studying the asymptotic behavior of jumping numbers and an extremal problem on a cone $\QM(Y,D)$ of quasi-monomial valuations. In \Cref{sec-i.v.i}, we prove \Cref{thm-infinite.interp} using valuation approximation, a strategy that more closely parallels the approach in \cite{BGMY25}. Finally, we demonstrate how \Cref{thm-finite.interp} recovers a result in \cite{BGMY25} (which deals with the complex analytic case) and yields an analytic characterization, derived from \Cref{thm-infinite.interp}, for the existence of infinite valuative interpolation.

\vspace{.1in} {\em Acknowledgements}. The first-named author completed this work during a visit to the School of Mathematical Sciences at Peking University and is grateful for its hospitality and support. The second-named author was supported by National Key R\&D Program of China 2021YFA1003100 and NSFC-12425101. The third-named author was supported by NSFC-12401099 and the Talent Fund of Beijing Jiaotong University 2024-004. The fourth-named author was supported by NSFC-12501106.

\section{Valuation and asymptotic Samuel function}\label{sec-r.o}

In this section, $R$ is a noetherian integral domain.

\subsection{Valuation}\label{subsec-val}

Recall that a function $v\colon R^*=R\setminus\{0\} \to \bbR_{\ge 0}$ is a (real) \emph{valuation} if
\[v(xy)=v(x)+v(y) \quad \& \quad v(x+y)\ge \min\{v(x),v(y)\},\]
for all $x,y\in R^*$. Set $v(0)=\infty$. A valuation $v$ on $R$ can be extended uniquely to a valuation $v\colon K\to \bbR\cup\{\infty\}$ on the field of fractions $K$ of $R$ by setting $v(x/y)=v(x)-v(y)$. Say $v\le w$ for two valuations $v,w$ on $R$ if $v(x)\le w(x)$ for all $x\in R^*$. For any nonzero ideal $\a$ of $R$, define $v(\a)=\inf_{x\in \a}v(x)$. Then it is easy to check: for nonzero ideals $\a,\b$ in $R$,
\[v(\a\b)=v(\a)+v(\b) \quad \& \quad v(\a+\b)= \min\{v(\a),v(\b)\}.\]
In particular, we can see if the ideals $\a_1,\ldots,\a_r$ in $R$ satisfies $\sum_{j=1}^r\a_j=R$, then
\[0=v(R)=v\big(\sum \a_j\big)=\min v(\a_j),\]
which implies $v(\a_j)=0$ for some $j\in\{1,\ldots,r\}$.

\subsection{Graded sequence and filtration of ideals}\label{subsec-gr.fil}

A \emph{graded sequence} of ideals $\a_{\blt}=(\a_m)_{m\in\bbZ_{>0}}$ is a sequence of ideals in $R$ that satisfies $\a_p\cdot \a_p\subseteq \a_{p+q}$ for all $p,q\ge 1$. Put $\a_0=R$. We always assume $\a_m\neq (0)$ for some $m\ge 1$. For example, $v$ is a nontrivial valuation on $R$, then $\a_{\blt}^v$ is a graded sequence of ideals in $R$ defined by $\a_{m}^v=\{x\in R \mid v(x)\ge m\}$. 

Now let $v$ be a valuation on $R$, and $\a_{\blt}$ a graded sequence of ideals in $R$. Then $v(\a_{p+q})\le v(\a_p)+v(\a_q)$ for all $p,q\ge 1$, which implies the limit
\[v(\a_{\blt})\coloneqq \lim_{m\to \infty, \, \a_m\neq (0)}\frac{v(\a_m)}{m}=\inf_{m\ge 1}\frac{v(\a_m)}{m}\]
exists by Fekete's lemma. In particular, $v(\a_{\blt}^v)=1$ for any nontrivial valuation $v$ on $R$. A graded sequence $\a_{\blt}=\{\a_m\}_{m\in\bbN}$ of ideals in $R$ is called a \emph{filtration} if additionally it satisfies $\a_m\supseteq \a_{m+1}$ for every $m\in\bbN$.

\subsection{Asymptotic Samuel function}

The classical asymptotic Samuel function (\cite{Sa52}) has been generalized to the filtration case in \cite{CP24}. Let us recall the constructions and some basic properties. Let $J_{\blt}$ be a filtration of ideals in $R$. For every ideal $\a$ in $R$, the \emph{order} of $\a$ with respect to $J_{\blt}$ is defined by
\[\vJ\big(\a\big)\coloneqq \sup\,\{m\in\bbN \mid \a\subseteq J_m\}\in\bbN\cup\{\infty\}.\]
Then we can see $\vJ(\a^{p+q})\ge \vJ(\a^p)+\vJ(\a^q)$ for every $p,q\in\bbN$. Therefore, it follows that
\[\vb(\a;J_{\blt})\coloneqq \lim_{m\to\infty}\frac{\vJ(\a^m)}{m}=\sup_{m\ge 1}\frac{\vJ(\a^m)}{m}\in \bbR_{\ge 0}\cup \{\infty\}\]
is well-defined by Fekete's lemma, called the \emph{asymptotic Samuel function} of $J_{\blt}$ (cf. \cite[Theorem 3.4]{CP24}). We also denote $\vb(\a;J_{\blt})\coloneqq \vbJ(\a)$ in the present paper.

\begin{example}\label{ex-asf}
    If $J_{\blt}=\{J^m\}_{m\ge 0}$ for a nonzero ideal $J$ in $R$, then $\vb(\a;J_{\blt})=\ol{\nu}_{J}(\a)$, which is the classical asymptotic Samuel function (cf. \cite{Sa52,HS06,LJT09}).
\end{example}

\begin{example}
    If we take $J_{\blt}=\a_{\blt}^v$ the filtration of valuation ideals associated to some valuation $v$ on $R$, then clearly $\vb(\a;J_{\blt})=v(\a)$.
\end{example}

The following lemma shows some basis properties of the asymptotic Samuel function, which are actually known in \cite{CP24}. The proof below follows from the strategy of \cite[Theorem 2]{Sa52}.

\begin{lemma}[{see \cite[Proposition 3.11]{CP24}}]\label{lem-mu.subadditivity}
    Let $J_{\blt}$ be a filtration of ideals in $R$. Then
    \begin{enumerate}
    \item $\vb(\a;J_{\blt})\ge \vb(\b;J_{\blt})$, for ideals $\a\subseteq\b$ in $R$;
    \item $\vb(\a^k;J_{\blt})=k\cdot\vb(\a;J_{\blt})$, where $\a$ is an ideal in $R$ and $k\in\bbZ_+$;
    \item $\vb(\a+\b;J_{\blt})= \min\{\vb(\a;J_{\blt}), \vb(\b;J_{\blt})\}$, where $\a,\b$ are ideals in $R$;
    \item $\vb(\a\b;J_{\blt})\ge \vb(\a;J_{\blt}) +\vb(\b;J_{\blt})$, where $\a,\b$ are ideals in $R$.
    \end{enumerate}
\end{lemma}

\begin{proof}
    The first two statements are obvious.
    
    (3): Clearly $\vb(\a+\b;J_{\blt})\le \min\{\vb(\a;J_{\blt}), \vb(\b;J_{\blt})\}$. Then we only need to prove the converse inequality. Without loss of generality, we assume both $\vb(\a;J_{\blt})$ and $\vb(\b;J_{\blt})$ are nonzero. Fix any $0<\sigma<\min\{\vb(\a;J_{\blt}), \vb(\b;J_{\blt})\}$. Then by definition we have the inclusions:
    \[\a^l\subseteq J_{\sigma l-F(l)}, \quad \b^l\subseteq J_{\sigma l-G(l)},\]
    where $F(l)=o(l)$ and $G(l)=o(l)$ as $l\to\infty$. Then it follows that
    \[(\a+\b)^l=\sum_{j=0}^l \a^j\b^{l-j}\subseteq \sum_{j=0}^l J_{\sigma j-F(j)}J_{\sigma (l-j)-G(l-j)}\subseteq J_{\sigma l-H(l)},\]
    where $H(l)=\sup_{0\le j\le l} \big(F(j)+G(l-j)\big)=o(l)$ as $l\to\infty$. Hence, we obtain $\vb(\a+\b; J_{\blt})\ge \sigma$, which gives the desired converse inequality by the choice of $\sigma$.

    (4): It suffices to prove the case where both $\vb(\a;J_{\blt})$ and $\vb(\b;J_{\blt})$ are nonzero. Let $0<\sigma_1<\vb(\a;J_{\blt})$ and $0<\sigma_2<\vb(\b;J_{\blt})$. Then we have $\a^l\subseteq J_{\sigma_1 l-o(l)}$ and  $\b^l\subseteq J_{\sigma_2 l-o(l)}$ as $l\to\infty$, which implies
    \[(\a\b)^{l}\subseteq J_{\sigma_1 l-o(l)}\cdot J_{\sigma_2 l-o(l)}\subseteq J_{(\sigma_1+\sigma_2)l-o(l)}.\]
    Hence, we obtain $\vb(\a\b; J_{\blt})\ge \sigma_1+\sigma_2$, which gives the desired inequality.
\end{proof}

Let $I$ be an ideal of $R$. The \emph{integral closure} of $I$ in $R$ is the ideal $\ol{I}$ consisting of $x\in R$ satisfying some monic polynomial equation:
\[x^n+a_1 x^{n-1}+\cdots+a_n=0, \quad \text{with} \ a_i\in I^i.\]
Note that if $\a,\b$ are ideals in $R$, then $\ol{\a}\cdot\ol{\b}\subseteq \ol{\a\cdot\b}$ (cf. \cite[Corollary 6.8.6]{HS06}). Thus, for a filtration $J_{\blt}$ of ideals in $R$, we have that $\ol{J}_{\blt}\coloneqq \{\ol{J_m}\}_{m\in\bbN}$ is also a filtration.

\begin{lemma}\label{prop-int.cl}
    For any filtration $J_{\blt}$ and any ideal $\a$ in $R$, we have $\vb(\a;J_{\blt})=\vb(\a;\ol{J}_{\blt})$.
\end{lemma}

\begin{proof}
    See \cite[Corollary 4.6]{CP24}.
\end{proof}

Now we consider the relations between valuations and asymptotic Samuel functions. We start from the following simple observation.

\begin{lemma}\label{lem-v(J.).le.mu.v(a)}
    Assume $\vb(\a;J_{\blt})<\infty$. For every valuation $v$ on $R$, we have
    \begin{equation}\label{eq-v.rho}
        v(\a)\ge \vb(\a;J_{\blt})v(J_{\bullet}).
    \end{equation}
\end{lemma}

\begin{proof}
    For every valuation $v$ on $R$, if $\a^l\subseteq J_m$ for $l\in\bbN$, then $v(\a)\cdot l=v(\a^l)\ge v(J_m)$. Dividing by $l$ in both sides and letting $l\to\infty$, we obtain $v(\a)\ge \vb(\a;J_{\blt})v(J_{\bullet})$.
\end{proof}

\begin{remark}\label{rem-asf.Rees}
Recall \Cref{ex-asf}. According to \cite[Lemma 10.1.5]{HS06}, if additionally $I$ is an ideal not contained in any minimal prime ideal of $R$, then the classical asymptotic Samuel function satisfies
\[\ol{\nu}_I(\a)=\min\Big\{\frac{v(\a)}{v(I)} \ \Big| \  v\in \mathcal{RV}(I)\Big\},\]
where $\mathcal{RV}(I)$ denotes the set of \emph{Rees valuations} of $I$. We see that in this case the equality in \eqref{eq-v.rho} must be attained by some Rees valuation. However, this fails for asymptotic Samuel functions of general filtration of ideals (see \Cref{rem-vb.neq.sigma}).
\end{remark}

Let $\a_{\blt}$ and $\b_{\blt}$ be graded sequences of ideals in $R$. Recall that $\a_{\blt}+\b_{\blt}$ is also a graded sequence of ideals defined by:
\[(\a_{\blt}+\b_{\blt})_m\coloneqq \sum_{i=0}^m \a_{i}\cdot \b_{m-i}, \quad m\in\bbN.\]

\begin{lemma}\label{lem-va.+b.}
     For every valuation $v$ on $R$,
     \[v(\a_{\blt}+\b_{\blt})=\min\{v(\a_{\blt}),v(\b_{\blt})\}.\]
\end{lemma}

\begin{proof}
    Denote $\fk{c}_{\blt}=\a_{\blt}+\b_{\blt}$. Since $\fk{c}_m\supseteq \a_m$ and $\fk{c}_m\supseteq \b_m$ for every $m$, we have $v(\fk{c}_{\blt})\le \min\{v(\a_{\blt}),v(\b_{\blt})\}$. On the other hand, for every $m\in\bbN$,
    \begin{equation*}
        \begin{aligned}
            v(\fk{c}_m)=\min_{0\le i\le m}v(\a_{i}\cdot \b_{m-i})&=\min_{0\le i\le m}\big\{v(\a_{i})+v(\b_{m-i})\big\}\\
            &\ge \min_{0\le i\le m}\big\{i\cdot v(\a_{\blt})+(m-i)\cdot v(\b_{\blt})\big\}\\
            &\ge m\cdot \min\{v(\a_{\blt}),v(\b_{\blt})\},
        \end{aligned}
    \end{equation*}
    which implies $v(\fk{c}_{\blt})\ge \min\{v(\a_{\blt}),v(\b_{\blt})\}$. We obtain the desired equality.
\end{proof}

Now let $\a_1,\ldots,\a_r$ be ideals in $R$, and $b_1,\ldots,b_r\in\bbR_{\ge 0}$. Recall the filtration
    \[I_{\blt}\coloneqq \Big(\frac{1}{b_1}\cdot \a_{1}\Big)_{\blt}+\cdots+\Big(\frac{1}{b_r}\cdot \a_{r}\Big)_{\blt}.\]

\begin{lemma}\label{lem-v(I.)}
    For every valuation $v$ on $R$, we have
    \[v(I_{\blt})=\min_{b_j>0}\left\{\frac{1}{b_1}v(\a_1), \ldots, \frac{1}{b_r}v(\a_r)\right\}.\]
\end{lemma}

\begin{proof}
    Note that $v\Big(\big(\frac{1}{b_j}\cdot \a_j\big)_{\blt}\Big)=\frac{1}{b_j}v(\a_j)$ for every $j$ with $b_j>0$. Then this lemma follows from \Cref{lem-va.+b.} and induction on $r$.
\end{proof}

\begin{lemma}\label{lem-mu.j.le.bj.-1}
    $\vb(\a_j;I_{\blt})\ge b_j$ for $j=1,\ldots, r$, and $\vb(\a_1\cdots\a_r;I_{\blt})\ge \sum_{j=1}^r b_j$.
\end{lemma}

\begin{proof}
    Note that $\a_j^{\lceil m/b_j \rceil}\subseteq \Big(\frac{1}{b_j}\cdot \a_{j}\Big)_m\subseteq I_m$ for all $m\ge 1$. We have $\vb(\a_j;I_{\blt})\ge b_j$ for $j=1,\ldots, r$. Then it also follows from \Cref{lem-mu.subadditivity} that $\vb(\a_1\cdots\a_r;I_{\blt})\ge \sum_{j=1}^r \vb(\a_j;I_{\blt})\ge \sum_{j=1}^r b_j$.
\end{proof}

The following proposition proves the implication ``(1) $\Rightarrow$ (3)'' in \Cref{thm-finite.interp}, where as we can see, this implication actually holds for all noetherian domains.

\begin{proposition}\label{prop-v(aj)=bj.implies.mu=bj-1}
    If $v$ is valuation on $R$ such that $v(\a_j)\ge b_j$ for $j=1,\ldots, r$, then $\vb(\a_1\cdots\a_r;I_{\blt})\le \sum_{j=1}^r v(\a_j)$. Moreover, if $v(\a_j) = b_j$ for all $j$, then $\vb(\a_1\cdots\a_r;I_{\blt})=\sum_{j=1}^r  b_j$.
\end{proposition}

\begin{proof}
    Since $v(\a_j)\ge b_j$ for all $j$, we have $v(I_{\blt})\ge 1$ by \Cref{lem-v(I.)}. Hence, it follows from \Cref{lem-v(J.).le.mu.v(a)} that $\vb(\a_1\cdots\a_r;I_{\blt})\le v(I_{\blt})^{-1}v(\a_1\cdots\a_r)\le \sum_{j=1}^r v(\a_j)$.
    
    If $v(\a_j) = b_j$ for all $j$, then the above inequality shows $\vb(\a_1\cdots\a_r;I_{\blt})\le \sum_{j=1}^r b_j$, while according to \Cref{lem-mu.j.le.bj.-1}, we have the converse inequality $\vb(\a_1\cdots\a_r;I_{\blt})\ge \sum_{j=1}^r b_j$. Hence, the desired equality holds.
\end{proof}

\section{Existence of finite valuative interpolation}\label{sec-f.v.i}

In this section, $R$ is an excellent regular domain of equicharacteristic $0$ and of dimension $n$. $X=\Spec R$.

\subsection{Divisorial valuation and quasi-monomial valuation}\label{sec-qm.log.dis.}
Let $X$ be a separated regular excellent scheme over $\bbQ$. Denote by $\Val_X$ the set of real valuations $v$ on the function field $K(X)$ of $X$ that admit a center $c_X(v)$ on $X$ (cf. \cite[Section 1.2]{JM12}).

A valuation $v\in\Val_X$ is called a \emph{divisorial valuation} on $X$, if there exists a prime divisor $E$ over $X$ (i.e. a prime divisor $E$ on a regular scheme $Y$ having a proper birational modification $\pi\colon Y\to X$), and a positive real number $\lambda\in\bbR_{>0}$, such that $v=\lambda\cdot \ord_E$. The \emph{log discrepancy} $A(v)$ of a divisorial valuation $v=\lambda\cdot\ord_E$ is defined such that $A(v)=\lambda\cdot A(\ord_E)$, where $A(\ord_E)\coloneqq\ord_E(K_{Y/X})+1$.

Let $\pi\colon Y\to X$ be a proper birational morphism with regular and connected $Y$, and let $y=(y_1,\cdots,y_r)\in \O_{Y,\eta}^{\oplus r}$ be a system of parameters at a point $\eta\in Y$. For $\alpha,\beta\in \bbR^r$, set $\langle\alpha,\beta\rangle\coloneqq \sum_{i=1}^r\alpha_i\beta_i$. The \emph{quasi-monomial valuation} associated with $y$ and $\alpha=(\alpha_1,\cdots,\alpha_r)\in \bbR^r_{\geq 0}$ is a valuation $\val_{\alpha}\colon K(X)\to \bbR\cup \{+\infty\}$ such that:
\begin{equation*}
    \val_{\alpha}(f)=\min\{\langle \alpha,\beta\rangle\colon c_{\beta}\neq 0\}
\end{equation*}
for all $f\in \O_{Y,\eta}$ written in $\wh{\O_{Y,\eta}}$ as  $f=\sum_{\beta\in \bbZ^r_{\geq 0}}c_{\beta}y^{\beta}$ with each $c_{\beta}\in\wh{\O_{Y,\eta}}$ either zero or a unit (cf. \cite[Section 3.1]{JM12}).

Let $(Y,D)$ be a log smooth pair over $X$. Denote by $\QM(Y,D)$ the set of all quasi-monomial valuations $v$ that can be described at some point $\eta\in Y$ as above such that each $y_i$ defines at $\eta$ an irreducible component of $D$. Then $\QM(Y,D)$ is the union of finitely many closed simplicial cones with each cone is homeomorphic to the cone $\bbR_{\ge 0}$ (see \cite[Lemma 4.5]{JM12}). There is also a \emph{retraction} map
\[r_{Y,D}\colon \Val_X \to \QM(Y,D),\]
which maps a valuation $v$ to the unique quasi-monomial valuation $w=r_{Y,D}(v)\in\QM(Y,D)$ with $w(D_i)=v(D_i)$ for every irreducible component $D_i$ of $D$. We have $r_{Y,D}(v)\le v$ and $r_{Y,D}(v)(\a)=v(\a)$ for an ideal on $X$ if $(Y,D)$ gives a log resolution of $\a$ (cf. \cite[Section 4.3]{JM12}).

The log discrepancy of a quasi-monomial $v\in\QM(Y,D)$ is defined by $A(v)=\sum_{i=1}^r v(D_i)A(\ord_{D_i})$,
and the log discrepancy of a general valuation $v\in\Val_X$ is defined by $A(v)=\sup_{(Y,D)}A(r_{Y,D}(v))\in\bbR_{\ge 0}\cup\{\infty\}$, where the supremum is over all log-smooth pair $(Y,D)$ over $X$ (cf. \cite[Section 5.2]{JM12}).

\subsection{Multiplier ideal and Skoda Theorem}

For an excellent scheme $X$ of characteristic $0$, the existence of \emph{log resolution} has been demonstrated in \cite{Te08,Te18}, generalizing Hironaka’s theorem: given an ideal $\a$ on $X$, there is a projective birational morphism $\pi\colon Y\to X$ such that $Y$ is regular, and $\a\cdot \O_Y$ is the ideal of a divisor $D$ such that $D+ K_{Y/X}$ is of simple normal crossing support.

Let $\a$ be a nonzero ideal on $X$, and $\lambda\in\bbR_{\ge 0}$. Then \emph{multiplier ideal} $\J(\lambda\cdot\a)$, also denoted by $\J(\a^{\lambda})$, is the ideal on $X$ consisting of local sections $f$ in $\O_X$ such that
\[\ord_E(f)+A(\ord_E)> \lambda \cdot \ord_E(\a)\]
for all prime divisors $E$ over $X$.

The following is the local version of an algebro-geometric analogue of Skoda's Division Theorem in complex analytic geometry (\cite{Sk72}). See also \cite[Variant 9.6.39]{La04} for Skoda's Theorem on singular normal varieties via log pairs.

\begin{theorem}[Skoda's Theorem]\label{thm-Skoda}
    For every nonzero ideal $\a$ on $X$ and $m\ge n$,
    \[\J(\a^m)\subseteq \a^{m-n+1}.\]
\end{theorem}

\begin{proof}
This theorem is well-established when $X$ is of finite type (see \cite[Theorem 9.6.21]{La04}). Moreover, it can be extended to regular excellent schemes by employing the same arguments in \cite[Appendix A]{JM12}.
\end{proof}

\subsection{Asymptotic multiplier ideal and asymptotic jumping number}

Let $J_{\blt}$ be a graded sequence of ideals on $X$. For every $\lambda\ge 0$, it can be seen that the following family of ideals on $X$
\[\Big\{\J\Big(\frac{\lambda}{p}\cdot J_{p}\Big)\Big\}_{p\in\bbZ_+}\]
admits a unique maximal element, denoted by $\J( J^{\lambda}_{\blt})$ or $\J(\lambda\cdot J_{\blt})$ which is defined to be the \emph{asymptotic multiplier ideal} of $J_{\blt}$ with exponent $\lambda$ (cf. \cite[Definition 11.1.15]{La04}).

Let $\q$ be a nonzero ideal on $X$. We denote by $\lct^{\q}(J_{\blt})$ or $\lct(\q;J_{\blt})$ the \emph{asymptotic jumping number} (cf. \cite[Proposition 2.12]{JM12}):
\[\lct^{\q}(J_{\blt})=\lct(\q;J_{\blt})\coloneqq \min\{\lambda\ge 0 \mid \q\not\subseteq \J(J_{\blt}^{\lambda})\},\]
or equivalently (see \cite[Corollary 6.9]{JM12}),
\begin{equation}\label{eq-compute}
    \lct^{\q}(J_{\blt})=\inf_{v}\frac{A(v)+v(\q)}{v(J_{\blt})},
\end{equation}
where the infimum is over all (divisorial, quasi-monomial) valuations on $X$. Jonsson and Musta\c{t}\u{a} proved in \cite{JM12} that $\lct^{\q}(J_{\blt})$ can be \emph{computed} by some valuation $v$, i.e. attaining the infimum in \eqref{eq-compute}.

\begin{theorem}[{\cite[Theorem 7.3]{JM12}}]\label{thm-compute}
    If $\lct^{\q}(J_{\blt})=\lambda<\infty$, then for every generic point $\xi$ of an irreducible component of $V(\J(J_{\blt}^{\lambda}):\q)$, there is a valuation $v\in\Val_X$ that computes $\lct^{\q}(J_{\blt})$, with $c_X(v)=\xi$.
\end{theorem}

Jonsson and Musta\c{t}\u{a} also conjectured in \cite{JM12} that the valuation computing the asymptotic jumping number can be chosen as a quasi-monomial valuation (see \cite[Conjecture C]{JM12}), where the cases $\dim X\le 2$ are proved by them in \cite{JM12}, and the case $\q=\O_X$ is proved by Xu in \cite{Xu20}. The general case still remains open.

\subsection{Asymptotic property of jumping numbers}

Let $\a$ be a proper ideal in $R$, and $J_{\blt}$ a graded sequence of nonzero ideals in $R$ with $\lct(J_{\blt})<\infty$. Note that it follows that $\lct(\a^k;J_{\blt})<\infty$ for all $k\in\bbN$. Furthermore, we can see $\lct(\a^k;J_{\blt})$ is concave in $k$ (cf. \cite[Lemma 6.1]{BGZ25}), and thus we can define
\[\rho(\a;J_{\blt})\coloneqq\lim_{k\to\infty}\frac{\lct(\a^k;J_{\blt})}{k}=\inf_{k\ge 1}\frac{\lct(\a^k;J_{\blt})}{k}\in [0,+\infty).\]

\begin{lemma}\label{lem-rho.-1.le.sigma}
    Suppose $J_{\blt}$ is filtration of ideals with $\lct(J_{\blt})<\infty$. Then $\vb(\a;J_{\blt}) \le\rho(\a;J_{\blt})$. In particular $\vb(\a;J_{\blt})<\infty$.
\end{lemma}

\begin{proof}
    May assume $\vb(\a;J_{\blt})>0$. Since $\J(I)\supseteq I$ for every ideal $I$ of $R$, for $m_k=\sup \{m \mid \a^k\subseteq J_m\}$, we have
    \[\a^{k}\subseteq J_{m_k} \subseteq \J(J_{m_k})\subseteq \J(m_k\cdot J_{\blt}) \implies \lct(\a^{k}; J_{\blt})\ge m_k.\]
    Hence,
    \[\rho(\a;J_{\blt})=\lim_{k\to\infty}\frac{\lct(\a^{k};J_{\blt})}{k}\ge \lim_{k\to\infty} \frac{m_k}{k} = \vb(\a;J_{\blt}) .\]
    We can also see $\vb(\a;J_{\blt})<\infty$ in this case since $\rho(\a;J_{\blt})<\infty$.
\end{proof}

\begin{remark}\label{rem-vb.neq.sigma}
    We can not expect $\vb(\a;J_{\blt}) =\rho(\a;J_{\blt})$ in general cases (see \cite[Example 7.2]{CP24}). The following is another example from \cite{KW07}.
    
    Let $R=\bbC[x,y]$ and $\m=(x,y)$ be the maximal ideal. Let $J_{\blt}$ the graded sequence: 
    \[J_k=\m^k \cdot (x^k, y), \quad \forall\, k\in\bbN.\]
    Then as shown in \cite[Example 3.6]{KW07}, for every $c>0$, the asymptotic multiplier ideal $\J(c\cdot J_{\blt})=\J(c\cdot \m)=\m^{\lfloor c \rfloor-1}$. It follows that $\rho(\m;J_{\blt})=1$ (one can also check this by the following \Cref{prop-sigma.equal.inf}). However, $\m^l\subseteq J_k$ if and only if $l\ge 2k$, which shows $\vb(\m;J_{\blt})=1/2\neq \rho(\m;J_{\blt})$.
    
    Furthermore, if $X$ is a smooth irreducible complex variety, and $J_{\blt}$ is a stable graded sequence in the sense of \cite[Definition 2.3]{KW07}, then according to \cite[Theorem 3.3]{KW07}, there is a suitable positive integer $C$ such that $\J(Ck\cdot J_{\blt})\subseteq J_k$ for all $k$ large enough, which implies that
    \[\vb(\a;J_{\blt}) \ge C^{-1}\rho(\a;J_{\blt})\]
    for any nonzero ideal $\a$ on $X$.
\end{remark}

In general, we cannot use the inequality \eqref{eq-v.rho} to characterize the quantity $\vb(\a;J_{\blt})$ conversely (e.g. seen by \Cref{rem-vb.neq.sigma} and the following \Cref{prop-sigma.equal.inf}), unlike the case of the classical asymptotic Samuel function in \Cref{rem-asf.Rees}. However, we will show in the following proposition that such a characterization does hold for $\rho(\a;J_{\blt})$ over $\Val_X^{<\infty}$ (and over the space of divisorial valuations on $X$), where
\[\Val_X^{<\infty}\coloneqq \{v\in\Val_X \mid A(v)<\infty\}.\]

\begin{proposition}\label{prop-sigma.equal.inf}
    Suppose $\lct(J_{\blt})<\infty$. Then (by convention setting $*/0=\infty$)
    \begin{equation}\label{eq-sigma.equal.inf}
        \rho(\a;J_{\blt})=\inf_{v\in\Val^{<\infty}_X}\frac{v(\a)}{v(J_{\blt})}=\inf_{E}\frac{\ord_E(\a)}{\ord_E(J_{\blt})},
    \end{equation}
    where the infimum in the last term is over all prime divisors $E$ over $X$.
\end{proposition}

\begin{proof}
    For every $v\in\Val^{<\infty}_X$, we have $\lct(\a^k;J_{\blt})\le \big(A(v)+k\cdot v(\a)\big)/v(J_{\blt})$, yielding that
    \[\rho(\a;J_{\blt})=\lim_{k\to\infty}\frac{\lct(\a^k;J_{\blt})}{k}\le \lim_{k\to\infty}\Big(\frac{v(\a)}{v(J_{\blt})}+\frac{1}{k}\cdot \frac{A(v)}{v(J_{\blt})}\Big)=\frac{v(\a)}{v(J_{\blt})}.\]
    On the other hand, let $v_k\in\Val_X^{<\infty}$ which computes $\lct(\a^k;J_{\blt})$. Then
    \[\lct(\a^k;J_{\blt})=\frac{A(v_k)+k\cdot v_k(\a)}{v_k(J_{\blt})}>k \cdot \frac{v_k(\a)}{v_k(J_{\blt})}\ge k \inf_{v\in\Val_X^{<\infty}}\frac{v(\a)}{v(J_{\blt})}.\]
    Hence, $\rho(\a;J_{\blt})\ge \inf_{v\in\Val_{X}^{<\infty}}v(\a)/v(J_{\blt})$. We get the first equality. The next equality follows from the continuity of the functions $v \mapsto v(\a)$ and $v \mapsto v(J_{\blt})$ on $\Val_X^{<\infty}$ (\cite[Corollary 6.4]{JM12}), and the denseness of divisorial valuations in $\Val_X^{<\infty}$ (\cite[Remark 4.11]{JM12}).
\end{proof}

It was also introduced in \cite[Definition 3.1]{BLQ24} the notion of the \emph{saturation} $\wt{J}_{\blt}$ of a filtration $J_{\blt}$:
\begin{equation}\label{eq-saturation}
    \wt{J}_{t}\coloneqq \{f\in R \mid \ord_{E}(f)\ge t\cdot \ord_E(J_{\blt}) \text{ for any prime divisor $E$ over $X$}\}
\end{equation}
for each $t\in\bbR_{\ge 0}$. Then the above \Cref{prop-sigma.equal.inf} is equivalent to say that $\rho(\a;J_{\blt})=\sup\,\{t\in\bbR_{\ge 0} \mid \a\subseteq \wt{J}_{t}\}$.

\begin{corollary}\label{cor-sigma.subadditivity}
    Suppose $\lct(J_{\blt})<\infty$. Then for nonzero ideals $\a,\b$ in $R$,
    \begin{equation*}
        \begin{aligned}
            \rho(\a+\b;J_{\blt})&=\min\{\rho(\a;J_{\blt}), \rho(\b;J_{\blt})\},\\
            \rho(\a\b;J_{\blt})&\ge \rho(\a;J_{\blt})+\rho(\b;J_{\blt}).
        \end{aligned}
    \end{equation*}
\end{corollary}

\begin{proof}
    By \Cref{prop-sigma.equal.inf}, we have
    \begin{equation*}
    \begin{aligned}
        \rho(\a+\b;J_{\blt})=\inf_{E}\frac{\ord_E(\a+\b)}{\ord_E(J_{\blt})}&=\inf_{E}\bigg(\min\bigg\{\frac{\ord_E(\a)}{\ord_E(J_{\blt})}, \frac{\ord_E(\b)}{\ord_E(J_{\blt})}\bigg\}\bigg)\\
        &=\min\bigg\{\inf_E\frac{\ord_E(\a)}{\ord_E(J_{\blt})}, \inf_E\frac{\ord_E(\b)}{\ord_E(J_{\blt})}\bigg\}\\
        &=\min\{\rho(\a;J_{\blt}), \rho(\b;J_{\blt})\},
    \end{aligned}
    \end{equation*}
    and
    \begin{equation*}
        \begin{aligned}
            \rho(\a\b;J_{\blt})=\inf_{E}\frac{\ord_E(\a\b)}{\ord_E(J_{\blt})}&= \inf_{E}\bigg(\frac{\ord_E(\a)}{\ord_E(J_{\blt})}+\frac{\ord_E(\b)}{\ord_E(J_{\blt})}\bigg)\\
            &\ge \inf_{E}\frac{\ord_E(\a)}{\ord_E(J_{\blt})}+\inf_{E}\frac{\ord_E(\b)}{\ord_E(J_{\blt})}\\
            &=\rho(\a;J_{\blt})+\rho(\b;J_{\blt}).
        \end{aligned}
    \end{equation*}
    The proof is complete.
\end{proof}

\subsection{A special case related to the finite valuative interpolation}

Recall that $\a_1,\ldots,\a_r$ are ideals of $R$, and $b_1,\ldots,b_r\in\bbR_{\ge 0}$, where $\sum_{b_j>0}\a_j\neq R$, i.e., $\bigcap_{b_j>0}V(\a_j)$ contains some closed point in $X$. Set
\[I_{\blt}\coloneqq \Big(\frac{1}{b_1}\cdot \a_{1}\Big)_{\blt}+\cdots+\Big(\frac{1}{b_r}\cdot \a_{r}\Big)_{\blt}.\]

\begin{lemma}\label{lem-sp.case.lct.finite}
    We have $\lct(I_{\blt})<\infty$ and $0<\vb(\a;I_{\blt})<\infty$ for $\a=\a_1\cdots\a_r$.
\end{lemma}

\begin{proof}
    Let $x\in\bigcap_{b_j>0}V(\a_j)$ be a closed point, i.e. $\a_j\subseteq \m_{x}$ for all $j$ with $b_j>0$, yielding
    \[\lct(I_{\blt})\le \frac{A(\ord_{x})}{\ord_{x}(I_{\blt})}=\frac{A(\ord_{x})}{\min_{b_j>0}\{\ord_{x}(\a_j)/b_j\}}<\infty.\]
    We can also see $\vb(\a;I_{\blt})>0$ by \Cref{lem-mu.j.le.bj.-1}. In addition, we have $\vb(\a;I_{\blt})<\infty$ due to \Cref{lem-rho.-1.le.sigma}.
\end{proof}

For this special filtration $I_{\blt}$, we will show that $\vb(\,\cdot\,; I_{\blt})$ coincides with $\rho(\,\cdot\,;I_{\blt})$, as demonstrated by the following two lemmas.

\begin{lemma}\label{lem-Skoda.multiplier.ideal}
Let $N\in\bbZ_+$, and let $d_j\in\bbN$ satisfy $N\ge b_jd_j$ for all $j$ with $b_j>0$. Then the asymptotic multiplier ideal of $I_{\blt}$ satisfies
\[\J( cN \cdot I_{\blt})\subseteq\J(c\cdot \b), \quad \forall\, c>0,\]
where $\b=\sum_{b_j>0} \a_j^{d_j}$.
\end{lemma}

\begin{proof}
    It suffices to prove
    \[\J\Big(\frac{cN}{p}\cdot I_{p}\Big)\subseteq \J(c\cdot \b), \quad \forall\, p\gg 1.\]
    Note that $N \lceil p_j/b_j\rceil\ge p_jd_j$ for all $j$ with $b_j>0$ and $p_j\in\bbN$. Thus, the integral closure
    \[\ol{I_p^{ N}}=\ol{\bigg(\sum_{p_1+\cdots+p_r=p} \prod_{b_j>0}\a_j^{\lceil p_j/b_j \rceil} \bigg)^{N}}\subseteq \ol{\sum_{p_1+\cdots+p_r=p} \prod_{b_j>0}\a_j^{p_jd_j}}=\ol{\b^p}.\]
    Above we used the fact: $\ol{(J_1+\cdots+J_r)^l}=\ol{J_1^l+\cdots+J_r^l}$ for $l\in\bbN$ and ideals $J_1,\ldots,J_r$ (cf. \cite[Example 1.1.2 \& Corollary 1.3.1]{HS06}). It follows that for every $p\gg 1$,
    \begin{equation*}
        \begin{aligned}
           \J\Big(\frac{cN}{p}\cdot I_{p}\Big)=\J\Big(\frac{c}{p}&\cdot I^{N}_{p}\Big)=\J\Big(\frac{c}{p}\cdot \ol{I_p^{N}}\Big)\\
           &\subseteq \J\Big(\frac{c}{p}\cdot \ol{\b^p}\Big)=\J\Big(\frac{c}{p}\cdot \b^p\Big)=\J(c\cdot\b), 
        \end{aligned}
    \end{equation*}
    by \cite[Corollary 9.6.17]{La04}. The proof is complete.
\end{proof}

\begin{lemma}\label{lem-finite.inter.case.rho.-1.equal.sigma}
    For every nonzero ideal $\a_0$ of $R$, we have $\vb(\a_0;I_{\blt})=\rho(\a_0;I_{\blt})$.
\end{lemma}

\begin{proof}
    We only need to prove $\vb(\a_0;I_{\blt})\ge\rho(\a_0;I_{\blt})$. Without loss of generality, we assume $\rho(\a_0;I_{\blt})>0$. Let $N\gg 1$ be a large enough positive integer, and $d_j\in\bbN$ such that
    \[\frac{d_j}{N}\le \frac{1}{b_j}\le \frac{d_j}{N}+\frac{1}{N}, \quad \forall\, j \text{ with } b_j>0.\]
    For every $l\in\bbZ_+$, let $k_l=\lceil \lct(\a_0^l;I_{\blt})/N \rceil-1$, i.e.
    \[N k_l<\lct(\a_0^l;I_{\blt})\le N(k_l+1),\]
    which shows $\a_0^l\subseteq \J(Nk_l \cdot I_{\blt})$. Meanwhile, by \Cref{lem-Skoda.multiplier.ideal}, $\J(Nk_l \cdot I_{\blt})\subseteq \J(\b^{k_l})$, where $\b=\sum_{b_j\ge 0} \a_j^{d_j}$. Now applying Skoda's Theorem (\Cref{thm-Skoda}) to $\b$, we get
    \[\J(\b^{k_l})\subseteq \b^{k_l-n+1}, \quad \forall\, k_l\ge n.\]
    In addition, since $\frac{1}{b_j}\le \frac{d_j}{N}+\frac{1}{N}$ for all $j$ with $b_j>0$, if let $e_N=\min_{b_j>0} \lfloor N-b_j \rfloor$, then we can see 
    \[I_{e_N}\supseteq \sum_{b_j>0} \a_j^{\lceil e_N/b_j\rceil}\supseteq \sum_{b_j\ge 0} \a_j^{d_j}=\b.\]
    In conclusion, when $k_l\ge n$, we have
    \[\a_0^{l}\subseteq\J(Nk_l \cdot I_{\blt})\subseteq\J(\b^{k_l})\subseteq\b^{k_l-n+1}\subseteq I_{e_N}^{k_l-n+1}\subseteq I_{(k_l-n+1)e_N}.\]
    This implies that
    \[\vb(\a_0;I_{\blt})\ge \limsup_{l\to \infty}\frac{(k_l-n+1)e_N}{l}=\frac{\rho(\a_0;I_{\blt})\cdot e_N}{N}, \quad \forall\, N\gg 1,\]
    where $ \rho(\a_0;I_{\blt})=\lim_{l\to \infty}Nk_l/l$ by the choice of $k_l$. Note that $N/e_N \to 1$ when $N\to \infty$, so we get $\vb(\a_0;I_{\blt})\ge \rho(\a_0;I_{\blt})$. The proof is complete.
\end{proof}

\begin{remark}
    The fact that $\vb(\a_0;I_{\blt})=\rho(\a_0;I_{\blt})$ for all $\a_0$ is equivalent to say
    \[\mathcal{K}(I_{\blt})=\wt{I}_{\blt},\]
    where $\wt{I}_{\blt}$ is the saturation of $I_{\blt}$ defined in \cite[Definition 3.1]{BLQ24} by \eqref{eq-saturation}, and the filtration $\mathcal{K}(I_{\blt})=\{K(I_{\blt})_t\}_{t\ge 0}$ is defined in \cite[Theorem 5.5]{CP24} by
    \[K(I_{\blt})_t\coloneqq \{f\in R \mid \vbI(f)\ge t\}, \quad \forall\,t\in\bbR_{\ge 0}.\]
\end{remark}

We prove the following lemma, where indeed the implication (3) $\Rightarrow$ (1) in \Cref{thm-finite.interp} is essentially due to this lemma, in combination with \Cref{lem-finite.inter.case.rho.-1.equal.sigma}.

\begin{lemma}\label{lem-finite.val.interp.exist.QM}
    Let $\a_0$ be a nonzero ideal of $R$ with $\vb(\a_0;I_{\blt})>0$. Suppose $(Y,D)$ is a log smooth pair over $X$ which gives a log resolution of $\a_1\cdots\a_r$. Then there exists a quasi-monomial valuation $v\in \QM(Y,D)$ such that
    \[\vb(\a_0;I_{\blt})=\frac{v(\a_0)}{v(I_{\blt})}.\]
\end{lemma}

\begin{proof}
    We have $\lct(I_{\blt})<\infty$ by \Cref{lem-sp.case.lct.finite} and hence $0<\vb(\a_0;I_{\blt})<\infty$. Let $r_{Y,D}\colon \Val\to \QM(Y,D)$ be the retraction map. Note that $r_{Y,D}(v)(\a_0)\le v(\a_0)$ and $r_{Y,D}(v)(\a_j)=v(\a_j)$ for $j=1,\ldots,r$, since $(Y,D)$ gives a log resolution of $\a_1\cdots\a_r$. Then it follows from
    \[\beta(v)\coloneqq\frac{v(\a_0)}{v(I_{\blt})}=\frac{v(\a_0)}{\min_{b_j>0}\{v(\a_j)/b_j\}}\]
    that $\beta(r_{Y,D}(v))\le \beta(v)$ for all $v\in\Val_X$, which implies ($*/0=\infty$)
    \[\inf_{v\in\Val^{<\infty}_X}\beta(v)=\inf_{v\in\QM(Y,D)}\beta(v)=\inf_{v\in\QM(Y,D), \, A(v)=1}\beta(v),\]
    where the last equality above holds because $\beta$ is invariant under rescaling. Since $\QM(Y,D)\cap\{v\mid A(v)=1\}$ is a compact subset of $\QM(Y,D)$ (cf. \Cref{sec-qm.log.dis.}), and $\beta(v)$ is continuous on this compact set, we can see that the infimum must be achieved by some $v\in\QM(Y,D)$, where the infimum exactly equals to $\vb(\a_0;I_{\blt})$ by \Cref{prop-sigma.equal.inf} and \Cref{lem-finite.inter.case.rho.-1.equal.sigma}.
\end{proof}

\subsection{Proof of \Cref{thm-finite.interp}}

Now we prove \Cref{thm-finite.interp}.

\begin{proof}[Proof of \Cref{thm-finite.interp}]
    (2) $\Rightarrow$ (1) is obvious, and (1) $\Rightarrow$ (3) is due to \Cref{prop-v(aj)=bj.implies.mu=bj-1}. It suffices to prove (3) $\Rightarrow$ (2). By \Cref{lem-finite.val.interp.exist.QM} and \Cref{lem-sp.case.lct.finite}, there exists a quasi-monomial valuation $v$ on $R$ such that $v(I_{\blt})=1$ and $v(\a)=\vb(\a;I_{\blt})=\sum_{j=1}^r b_j$. It follows from $v(I_{\blt})=1$ that $v(\a_j)\ge b_j$ for $j=1,\ldots, r$, due to \Cref{lem-v(I.)}. Hence, since $v(\a)=\sum_{j=1}^r v(\a_j)$, we get $v(\a_j)=b_j$ for all $j=1,\ldots, r$. The proof is complete. 
\end{proof}

It is natural to ask whether $\vbI(\a_j) = b_j$ for all $j = 1, \ldots, r$ suffices to ensure the existence of a valuation $v$ such that $v(\a_j) = b_j$ for all $j$. However, the following example demonstrates that considering the asymptotic Samuel function at the product $\a = \a_1 \cdots \a_r$ is essential.

\begin{example}
Consider $R=\bbC[x,y]$. Let $\a_1=\m^6+(x)$, $\a_2=\m^6+(y)$ and $\a_3=\m^4+(xy)$, where $\m=(x,y)$. Let $b_1=1$, $b_2=1$ and $b_3=3$. Then one can check that $\vbI(\a_1)=1$, $\vbI(\a_2)=1$, $\vbI(\a_3)=3$ and $\vbI(\a_1\a_2\a_3)=6$. Hence, even if $\a_1$, $\a_2$ and $\a_3$ are all $\m$-primary, we can not deduce $\vbI(\a_1\a_2\a_3)=b_1+b_2+b_3$ from $\vb(\a_j;I_{\blt})=b_j$ for $j=1,2,3$.
\end{example}

\section{Existence of infinite valuative interpolation}\label{sec-i.v.i}

In this section, $R$ is an excellent regular domain of equicharacteristic $0$ and of dimension $n$. $X=\Spec R$.

\subsection{Valuation approximation method}

To characterize the existence of an infinite valuative interpolation, we first investigate whether the infimum in the middle term of \eqref{eq-sigma.equal.inf} is attained by some valuation in $\Val_X^{<\infty}$. We use the valuation approximation method, which is closer to the approach in \cite{BGMY25}.

Let $J_{\blt}$ be a filtration of ideals in $R$.

\begin{lemma}\label{lem-exist.Afinite.implies.sup.finite}
    Suppose $\lct(J_{\blt})<\infty$. Let $\a$ be a nonzero ideal of $R$. If there exists $v\in\Val^{<\infty}_X$ such that $\rho(\a;J_{\blt})=v(\a)/v(J_{\blt})$, then
    \[\sup_{k\in\bbN}\Big(\lct(\a^k; J_{\blt})-\rho(\a;J_{\blt})\cdot k\Big)<\infty.\]
\end{lemma}

\begin{proof}
May assume $v(J_{\blt})=1$ after rescaling. Then for every $k\in\bbN$,
    \[\lct(\a^k;J_{\blt})\le\frac{A(v)+v(\a^k)}{v(J_{\blt})}=A(v)+\rho(\a;J_{\blt})\cdot k,\]
    which implies $\sup_{k\in\bbN}\Big(\lct(\a^k; J_{\blt})-\rho(\a;J_{\blt})\cdot k\Big)\le A(v)<\infty$.
\end{proof}

Let $\xi\in X$ be a point, and $\m_{\xi}$ the ideal defining the closure of $\xi$. Let $\ord_{\xi}$ be the associated order valuation with center $\xi$:
\[\ord_{\xi}(f)\coloneqq \max\big\{k\in\bbN \mid f\in \m_{\xi}^{k}\big\}.\]
Set
\[\Val_{X,\xi}^{<\infty}\coloneqq \{v\in\Val_X\mid c_X(v)=\xi, \ A(v)<\infty\}.\]

\begin{proposition}\label{prop-val.interp.gr.seq}
    Assume $\ord_{\xi}(J_{\blt})>0$ and $\m_{\xi}^p\subseteq J_{q}$ for some $p,q\in\bbZ_+$. Let $\a$ be a nonzero ideal of $R$ with $\a\subseteq\m_{\xi}$. Then the following two statements are equivalent:
    \begin{enumerate}
        \item There exists $v\in\Val_{X,\xi}^{<\infty}$ such that $v(J_{\blt})=1$ and $v(\a)=\rho(\a;J_{\blt})$;
        \item It holds that
        \[M\coloneqq \sup_{k\in\bbN}\Big(\lct(\a^k;J_{\blt})-\rho(\a;J_{\blt})\cdot k\Big)<\infty.\]
    \end{enumerate}  
    Moreover, the valuation $v$ in (1) can be chosen with $A(v)\le M$ and $v(\m_{\xi})\ge q/p$ if (2) holds.
\end{proposition}

\begin{proof}
    First, observe $\lct(J_{\blt})<\infty$ by 
    \[\lct(J_{\blt})=\inf_{v\in\Val_X^*}\frac{A(v)}{v(J_{\blt})}\le\frac{A(\ord_{\xi})}{\ord_{\xi}(J_{\blt})}<\infty.\]
    It follows that $\lct(\a^k; J_{\blt})<\infty$ for all $k\in\bbN$ and $\rho(\a;J_{\blt})<\infty$.
    
    (1) $\Rightarrow$ (2): This follows from \Cref{lem-exist.Afinite.implies.sup.finite} immediately.
    
    (2) $\Rightarrow$ (1): We now assume
    \[M=\sup_{k\in\bbN}\Big(\lct(\a^k; J_{\blt})-\rho(\a;J_{\blt})\cdot k\Big)<\infty.\]
    For every $k\in\bbN$, since $J_q\supseteq \m_{\xi}^p$ and $\a\subseteq\m_{\xi}$, there is a valuation $v_k\in\Val^{<\infty}_{X,\xi}$ which computes $\lct(\a^k;J_{\blt})$ and satisfies $v_k(J_{\blt})=1$ (see \Cref{thm-compute}).
    Then we have
    \begin{equation}\label{eq-vk.compute}
        \lct(\a^k;J_{\blt})=\frac{A(v_k)+v_k(\a^k)}{v_k(J_{\blt})}= A(v_k)+k v_k(\a).
    \end{equation}
    Since $v_k(\a)\ge \rho(\a;J_{\blt})v_k(J_{\blt})=\rho(\a;J_{\blt})$ due to \Cref{prop-sigma.equal.inf}, it follows that
    \[\sup_{k\in\bbN}A(v_k)\le \sup_{k\in\bbN}\Big(\lct(\a^k; J_{\blt})-\rho(\a;J_{\blt})\cdot k\Big)= M<\infty.\]
    Meanwhile, we have
    \[\m_{\xi}^p\subseteq J_q \implies v_k(\m_{\xi})\ge p^{-1}v_k(J_q)\ge \frac{q}{p}v_k(J_{\blt})=\frac{q}{p}>0.\]
    Therefore, as the set
    \[\big\{w\in\Val_X\mid c_X(w)=\xi, \ w(\m_{\xi})\ge q/p, \ A(w)\le M\big\}\]
    is compact in $\Val_X$ (see \cite[Proposition 5.9]{JM12}), we can extract a subsequence of $(v_k)_{k\in\bbN}$ which is convergent to some $v\in\Val_{X,\xi}^{<\infty}$, with $A(v)\le M$ and $v(\m_{\xi})\ge q/p$. According to \cite[Corollary 6.4]{JM12}, we have $v(J_{\blt})=1$. We also obtain from (\ref{eq-vk.compute}) that
    \[v(\a)=\lim_{k\to\infty}v_k(\a)=\lim_{k\to\infty}\frac{\lct(\a^k;J_{\blt})}{k}=\rho(\a;J_{\blt}).\]
    Hence, the valuation $v$ is just what we need.
\end{proof}

\begin{example}\label{ex-rho.-1.neq.sigma}
Let $R=\bbC[x,y]$ and $\m=(x,y)$. Let $v$ be a normalized valuation on $R$ centered at $o$ which is \emph{infinite singular} with finite \emph{skewness} $\alpha(v)<\infty$ and infinite log discrepancy $A(v)=\infty$ (see \cite[Remark A.4]{FJ04} for the existence of such valuations). Set $J_{\blt}=\a_{\blt}^v$ to be the graded sequence associated to the valuation $v$. One can verify $\lct(\a_{\blt}^v)<\infty$ and $\rho(\m;\a_{\blt}^v)=1$ via the valuative tree theory.

Then $v$ itself is the (unique) valuation satisfying $v(\m)=1=\rho(\m;\a_{\blt}^v)$ and $v(\a_{\blt}^v)=1$, while $v\notin\Val_{X}^{<\infty}$. Indeed, we have
\[\sup_{k\in\bbN}\big(\lct(\m^k;\a_{\blt}^v)-k\big)=A(v)=\infty\]
for this case according to \cite[Proposition A.2]{BGZ25}.
\end{example}

\Cref{prop-val.interp.gr.seq} can be generalized to countably infinite ideals case. We start from the following observation.

\begin{lemma}\label{lem-multiply.change.to.plus}
    Let $J_{\blt}$ be a  filtration of nonzero ideals in $R$ with $\lct(J_{\blt})<\infty$, and let $\a_1,\ldots,\a_r$ be nonzero ideals in $R$. Set $\a=\a_1\cdots\a_r$. Suppose there is a valuation $v\in\Val_X^{<\infty}$ such that $v(\a)=\rho(\a;J_{\blt})$ and $v(J_{\blt})=1$. Then $v(\a_j)=\rho(\a_j;J_{\blt})$ for $j=1,\ldots,r$ if and only if $\rho(\a;J_{\blt})=\sum_{j=1}^r \rho(\a_j;J_{\blt})$.
\end{lemma}

\begin{proof}
    ``If" part: Suppose $\rho(\a;J_{\blt})=\sum_{j=1}^r \rho(\a_j;J_{\blt})$. Then 
    \[\sum_{j=1}^r v(\a_j)=v(\a)=\rho(\a;J_{\blt})=\sum_{j=1}^r \rho(\a_j;J_{\blt}),\]
    while each $v(\a_j)\ge v(J_{\blt})\rho(\a_j;J_{\blt})=\rho(\a_j;J_{\blt})$ due to \Cref{prop-sigma.equal.inf}. Hence, $v(\a_j)=\rho(\a_j;J_{\blt})$ for $j=1,\ldots,r$.
    
    ``Only if" part: This is a direct consequence of $v(\a)=\sum_{j=1}^r v(\a_j)$.
\end{proof}

Now we prove the following generalization of \Cref{prop-val.interp.gr.seq}, where the finite case will be used in the proof of \Cref{thm-infinite.interp}.

\begin{proposition}\label{prop-gr.seq.infinite.compute}
    Let $J_{\blt}$ be a  filtration of ideals in $R$ such that $\ord_{\xi}(J_{\blt})>0$ and $\m_{\xi}^p\subseteq J_{q}$ for some $p,q\in\bbZ_+$. Assume $(\a_j)_{j\in\bbZ_+}$ is a sequence of nonzero ideals of $R$ with $\a_j\subseteq\m_{\xi}$ for all $j\ge 1$. Then the following two statements are equivalent:
    \begin{enumerate}
        \item There exists $v\in\Val_{X,\xi}^{<\infty}$ such that $v(J_{\blt})=1$ and $v(\a_j)=\rho(\a_j;J_{\blt})$ for all $j\ge 1$;
        \item For every $r\in\bbZ_+$,
        \[\rho(\a_1\cdots\a_r;J_{\blt})=\sum_{j=1}^r\rho(\a_j;J_{\blt}),\]
        and it holds that
        \[\sup_{r\ge 1}\sup_{k\in\bbN}\Big(\lct(\a_1^k\cdots\a_r^k; J_{\blt})-\rho(\a_1\cdots\a_r;J_{\blt})\cdot k\Big)<\infty.\]
    \end{enumerate}
\end{proposition}

\begin{proof}
    First, as in the proof of \Cref{prop-val.interp.gr.seq}, we can see $\lct(J_{\blt})<\infty$ and $\rho(\a;J_{\blt})<\infty$ for all nonzero ideals $\a$ of $R$. 
    
    (1) $\Rightarrow$ (2): By the assumption and \Cref{prop-sigma.equal.inf}, for every $r\in\bbZ_+$,
    \begin{equation*}
        \begin{aligned}
            \sum_{j=1}^r \rho(\a_j;J_{\blt})=\sum_{j=1}^r v(\a_j)=v(\a_1\cdots\a_r)&\ge \rho(\a_1\cdots\a_r;J_{\blt})\cdot v(J_{\blt})\\
            &=\rho(\a_1\cdots\a_r;J_{\blt}),
        \end{aligned}
    \end{equation*}
    so we have $\rho(\a_1\cdots\a_r;J_{\blt})=\sum_{j=1}^r\rho(\a_j;J_{\blt})$ by \Cref{cor-sigma.subadditivity}, which also yields that $v(\a_1\cdots\a_r)=\rho(\a_1\cdots\a_r;J_{\blt})$ for any $r\ge 1$.
    
    Now, for every $k\in\bbN$ and $r\ge 1$,
    \[\lct(\a_1^k\cdots\a_r^k; J_{\blt})\le\frac{A(v)+v(\a_1^k\cdots\a_r^k)}{v(J_{\blt})}=A(v)+\rho(\a_1\cdots\a_r;J_{\blt})\cdot k,\]
    which implies \[\sup_{r\ge 1}\sup_{k\in\bbN}\Big(\lct(\a_1^k\cdots\a_r^k; J_{\blt})-\rho(\a_1\cdots\a_r;J_{\blt})\cdot k\Big)\le A(v)<+\infty.\]
    
    (2) $\Rightarrow$ (1): We now assume $\rho(\a_1\cdots\a_r;J_{\blt})=\sum_{j=1}^r\rho(\a_j;J_{\blt})$ for all $r\ge 1$, and
    \[M\coloneqq \sup_{r\ge 1}\sup_{k\in\bbN}\Big(\lct(\a_1^k\cdots\a_r^k; J_{\blt})-\rho(\a_1\cdots\a_r;J_{\blt})\cdot k\Big)<+\infty.\]
    According to \Cref{prop-val.interp.gr.seq}, for any $r\ge 1$, we can find a valuation $v_r\in \Val_{X,\xi}^{<\infty}$ such that $v_r(J_{\blt})=1$ and $v_r(\a_1\cdots\a_r)=\rho(\a_1\cdots\a_r;J_{\blt})$. Note that since $\rho(\a_1\cdots\a_r;J_{\blt})=\sum_{j=1}^{r}\rho(\a_j;J_{\blt})$ for all $r$, it follows from \Cref{lem-multiply.change.to.plus} that $v_r(\a_j)=\rho(\a_j;J_{\blt})$ for all $1\le j\le r$. Moreover, we can assume $v_r(\m_{\xi})\ge q/p$ and $A(v_r)\le M$ for all $r$ by \Cref{prop-val.interp.gr.seq}. Hence, there exists a subsequence of $(v_r)_{r\in\bbZ_+}$ which converges to some $v\in\Val_{X,\xi}^{<\infty}$ (which also satisfies $v(\m_{\xi})\ge q/p$ and $A(v)\le M$). We can see that $v(\a_j)=\rho(\a_j;J_{\blt})$ for all $j\in\bbZ_+$, and by \cite[Corollary 6.4]{JM12} we have $v(J_{\blt})=1$ since every $v_r(J_{\blt})=1$.
    
    The proof is complete.
\end{proof}

\subsection{Proof of \Cref{thm-infinite.interp}}

\begin{proof}[Proof of \Cref{thm-infinite.interp}]
    First, we prove (1) $\Rightarrow$ (2). Assume $v$ is a valuation on $R$ with $A(v)<\infty$ such that $v(\a_j)=b_j$ for all $j\in\bbZ_+$. Then $\vb(\a_1\cdots\a_r;I_{\blt}^{(r)})=\sum_{j=1}^r b_j$ for all $j\in\bbZ_+$ by \Cref{prop-v(aj)=bj.implies.mu=bj-1}. In addition, since $v(I_{\blt}^{(r)})=\min_{1\le j\le r}\{v(\a_j)/b_j\}=1$, for every $r\ge 1$ and $k\in\bbN$, we have
    \[\lct(\a_1^k\cdots\a_r^k;I_{\blt}^{(r)})\le \frac{A(v)+v(\a_1^k\cdots\a_r^k)}{v(I_{\blt}^{(r)})}=A(v)+k\sum_{j=1}^r b_j,\]
    which confirms $\sup_{r}\sup_{k}\Big(\lct\big(\a_1^k\cdots\a_r^k; I^{(r)}_{\blt}\big)-k\sum_{j=1}^r b_j\Big)\le A(v)<\infty$.

    Next, we prove (2) $\Rightarrow$ (1). By assumption, there exists some $r_0\in\bbZ_+$ such that $\sqrt{\sum_{j=1}^{r_0} \a_{j}}=\m$, which implies that there exists $s\in\bbZ_+$ such that $\m^{s}\subseteq \sum_{j=1}^{r_0}\a_j\subseteq \m$ (cf. \cite[Corollary 7.16]{AM69}). Hence, by taking $p= s\cdot \sum_{j=1}^{r_0} \lceil 1/b_j \rceil$, we have
    \[I_{1}^{(r)}\supseteq \m^{p}, \quad \forall\, r\ge r_0.\]
    Moreover, $\ord_{\m}(I^{(r)}_{\blt})=\min_{1\le j\le r}\{\ord_{\m}(\a_j)/b_j\}>0$ for all $r$.
    
    Let $\fk{c}_r\coloneqq \a_1\cdots\a_r$ for $r\ge r_0$. Due to \Cref{lem-finite.inter.case.rho.-1.equal.sigma}, we have $\rho(\fk{c}_r;I^{(r)}_{\blt})=\vb(\fk{c}_r;I^{(r)}_{\blt})=\sum_{j=1}^r b_j$ for all $r\ge 1$. Now by \Cref{prop-val.interp.gr.seq}, for every $r\ge r_0$, there exists a valuation $v_r\in\Val_{X}$ on $R$ such that $v_r(I_{\blt}^{(r)})=1$ and $v_r(\fk{c}_r)=\rho(\fk{c}_r;I^{(r)}_{\blt})=\sum_{j=1}^r b_j$. Moreover, the valuation $v_r$ can satisfy $v_r(\m)\ge 1/p$ and
    \[A(v_r)\le \sup_{k\ge 1}\Big(\lct(\fk{c}_r;I^{(r)}_{\blt})-k\sum_{j=1}^r b_j\Big)\le M,\]
    where $M\coloneqq \sup_{r\ge 1}\sup_{k\in\bbN}\Big(\lct\big(\a_1^k\cdots\a_r^k; I^{(r)}_{\blt}\big)-k\sum_{j=1}^r b_j\Big)<\infty$. We verify $v_r(\a_j)=b_j$ for all $1\le j\le r$. In fact, we have $\sum_{j=1}^r v_r(\a_j)=v_r(\fk{c}_r)=\sum_{j=1}^r b_j$, and $v_r(\a_j)\ge \vb(\a_j;I_{\blt}^{(r)})\cdot  v_r(I_{\blt}^{(r)})\ge b_j$ for $1\le j\le r$ by \Cref{lem-mu.j.le.bj.-1}. It follows that $v_r(\a_j)=b_j$ for all $1\le j\le r$. Meanwhile, since every $v_r$ lies in the following compact subset of $\Val_X$,
    \[\{w\in\Val_X \mid A(w)\le M, \ w(\m)\ge 1/p\},\]
    we can extract a subsequence of $(v_r)_{r\ge r_0}$ which converges to some valuation $v$ on $R$ with $A(v)\le M<\infty$. Since every $v_r$ satisfies $v_r(\a_j)=b_j$ for $1\le j\le r$, we can see $v(\a_j)=b_j$ for all $j\in\bbZ_+$. Therefore, the valuation $v$ is just what we need.
    \end{proof}

\section{Complex analytic case}\label{sec-c.a.c}

In this section, $R=\O_{\bbC^n,o}$ is the ring of germs of holomorphic functions at the origin $o$ in $\bbC^n$.

For every nonzero ideal $\a$ of $R$, we can associate to $\a$ a plurisubharmonic (psh for short) function near $o$ as follows: if $g_1,\ldots,g_s$ is a set of generators of $\a$, then we set $\log|\a|\coloneqq \log\big(|g_1|^2+\cdots+|g_s|^2\big)^{1/2}$. Note that this definition depends on the choice of generators for $\a$, but any two such choices yield psh functions that differ by a bounded term near $o$. For ideals $\a_1,\a_2$ of $R$, near $o$, we have
\begin{equation*}
    \begin{split}
        \log|\a_1\a_2|&=\log|\a_1|+\log|\a_2|+O(1),\\ \quad \log|\a_1+\a_2|&=\max\{\log|\a_1|,\log|\a_2|\}+O(1).
    \end{split}
\end{equation*}
Notably, for ideals $\a$ and $I$ of $R$, we also have $\log|\a|\le \log|I|+O(1)$ if and only if $\a\subseteq \ol{I}$, where $\ol{I}$ is the integral closure of $I$ in $R$; see e.g. \cite[Corollary 10.5]{CADG}.

Now suppose $\a_1,\ldots,\a_r$ are proper ideals of $R$, and $b_1,\ldots,b_r\in\bbR_{+}$. Define the psh function
\[\Phi=\log\Big(|\a_1|^{1/b_1}+\cdots+|\a_r|^{1/b_r}\Big)=\max_{1\le j\le r}\log|\a_j|^{1/b_j}+O(1)\]
near $o$, and for every psh function $\psi$ near $o$ recall that the \emph{relative type} \cite{Ras06} of $\psi$ to $\Phi$ at $o$ is defined by
\[\sigma(\psi,\Phi)\coloneqq \sup\,\{c\ge 0 \mid \psi\le c\Phi+O(1) \ \text{near} \ o\}.\]
Let $I_{\blt}\coloneqq \Big(\frac{1}{b_1}\cdot \a_{1}\Big)_{\blt}+\cdots+\Big(\frac{1}{b_r}\cdot \a_{r}\Big)_{\blt}$.

\begin{lemma}\label{lem-log.Im.le.mvarphi}
    For every $m\in\bbZ_+$, we have
    \[\big(m+\max_{1\le j\le r} b_j\big)\Phi+O(1)\le \log|I_m|\le m\Phi+O(1) \quad \text{near} \ o.\]
\end{lemma}

\begin{proof}
Near $o$, we have
\begin{equation*}
    \begin{aligned}
        \log|I_m|&=\log\bigg|\sum_{m_1+\cdots+m_r=m} \a_1^{\lceil m_1/b_1\rceil}\cdots \a_r^{\lceil m_r/b_r\rceil}\bigg|+O(1)\\
        &=\max_{m_1+\cdots+m_r=m}\bigg(\sum_{j=1}^r\lceil m_j/b_j\rceil \log|\a_j|\bigg)+O(1),
    \end{aligned}
\end{equation*}
which implies
    \begin{equation*}
        \begin{aligned}
            \log|I_m|\le \max_{m_1+\cdots+m_r=m}\bigg(&\sum_{j=1}^r m_j\cdot\log|\a_j|^{1/b_j}\bigg)+O(1)\\
            &= m \cdot \max_{1\le j\le r}\log|\a_j|^{1/b_j}+O(1)=m\Phi+O(1),
        \end{aligned}
    \end{equation*}
    and
    \begin{equation*}
        \begin{aligned}
            \log|I_m|&\ge \max_{1\le j\le r}\lceil m/b_j \rceil \log|\a_j|+O(1)\\
            &\ge \max_{1\le j\le r}\big( m/b_j+1\big) \log|\a_j|+O(1) \ge \big(m+\max_{1\le j\le r} b_j\big)\Phi+O(1).
        \end{aligned}
    \end{equation*}   
\end{proof}

We show that the relative type and jumping number of $\Phi$ coincide with some algebraic notions of $I_{\blt}$ in the following two lemmas respectively. 

\begin{lemma}\label{lem-analytic.case.rho-1.equal.sigma}
    $\vb(\a;I_{\blt})=\sigma(\log|\a|,\Phi)$ for every ideal $\a$ of $R=\O_{\bbC^n,o}$.
\end{lemma}

\begin{proof}
    As $\vb(\a;I_{\blt})=\vb(\a;\ol{I}_{\blt})$ by \Cref{prop-int.cl}, we only need to show $\vb(\a;I_{\blt})\le \sigma(\log|\a|,\Phi)$ and $\vb(\a;\ol{I}_{\blt})\ge \sigma(\log|\a|,\Phi)$.
    
    First, we prove $\vb(\a;I_{\blt})\le \sigma(\log|\a|,\Phi)$. May assume $\vb(\a;I_{\blt})>0$. For every positive integer $l$, let $m_l=\max\,\{m\in\bbN\mid \a^{l}\subseteq I_m\}$. Then it follows from \Cref{lem-log.Im.le.mvarphi} that
\[l\cdot\log|\a|\le \log|I_{m_l}|+O(1)\le m_l\cdot\Phi+O(1).\]
    By the definition of relative type, we have
    \[\sigma(\log|\a|,\Phi)\ge \sup_{l\in\bbZ_+}\frac{m_l}{l}=\vb(\a;I_{\blt}).\]
    
    Next, we prove $\vb(\a;\ol{I}_{\blt})\ge \sigma(\log|\a|,\Phi)$. May assume $\sigma(\log|\a|,\Phi)>0$. Given any positive real number $\sigma<\sigma(\log|\a|,\Phi)$, we have $\log|\a|\le \sigma\Phi+O(1)$ near $o$. For $L\in\bbZ_+$ sufficiently large, clearly there exists positive integers $d_1,\ldots,d_r$ such that
    \[\frac{d_j}{L}\le \frac{1}{b_j}< \frac{d_j}{L}+\frac{1}{L}, \quad j=1,\ldots r.\]
    It follows that
    \begin{equation*}
        \begin{split}
            \lceil L/\sigma \rceil\cdot\log|\a|\le L\max_{1\le j\le r}\log|\a_j|^{1/b_j}+O(1)&\le L\cdot \max_{1\le j\le r}\log|\a_j|^{d_j/L}+O(1)\\
            &\le \log\Big(\sum_{1\le j\le r}\big|\a_j^{d_j}\big|\Big)+O(1).
        \end{split}
    \end{equation*}
    Let $J=\sum_{j=1}^r \a_j^{d_j}$ be an ideal of $\O_{\bbC^n,o}$. Then the above inequality says that $\log\big|\a^{\lceil L/\sigma \rceil}\big|\le \log |J|+O(1)$, which indicates that $\a^{\lceil L/\sigma \rceil}\subseteq \ol{J}$. Now we verify that $J\subseteq I_m$ for $m= \lfloor L- \max_{1\le j\le r} b_j\rfloor$. In fact, for such $m$, we have \[m/b_j \le \Big(\frac{1}{b_j}-\frac{1}{L}\Big)L\le d_j \implies \lceil m/b_j\rceil \le d_j, \quad j=1,\dots, r,\]
    which implies $J=\sum\a_j^{d_j}\subseteq I_m$. Thus, we have $\a^{\lceil L/\sigma \rceil}\subseteq\ol{J}\subseteq \ol{I_m}$ for $m= \lfloor L- \max_{1\le j\le r} b_j\rfloor$. It follows that 
    \[\vb(\a;\ol{I}_{\blt})\ge \sup_{L}\frac{ \lfloor L- \max_{1\le j\le r} b_j\rfloor}{\lceil L/\sigma \rceil}=\sigma.\]
    Taking $\sigma \to \sigma(\log|\a|,\Phi)$, we get $\vb(\a;\ol{I}_{\blt})\ge \sigma(\log|\a|,\Phi)$.
    
    The proof is complete.
\end{proof}

For a psh function $\phi$ near $o$, the \emph{multiplier ideal} of $\phi$ at $o$ is defined by:
\[\I(\phi)_o\coloneqq \big\{ f\in \O_{\bbC^n,o} \mid |f|^2e^{-2\phi} \text{ is integrable near } o\},\]
and the \emph{jumping number} $c_o^{\q}(\phi)\coloneqq \sup\,\{c\ge 0 \mid \q\subseteq \I(c\phi)_o\}$ for every ideal $\q$ of $R=\O_{\bbC^n,o}$. In particular, if $\phi=t\log|\a|$ for some nonzero ideal $\a\subseteq R$ and $t>0$, then $\I(t\log|\a|)_o=\J(t\cdot \a)$ (see \cite{DK01}).

\begin{lemma}\label{lem-monomial.case}
    $\lct^{\q}(I_{\blt})=c_o^{\q}(\Phi)$ for every nonzero ideal $\q$ of $R=\O_{\bbC^n,o}$.
\end{lemma}

\begin{proof}
    Let $c>0$. By \Cref{lem-log.Im.le.mvarphi}, we have for every $m\in\bbZ_+$,
    \[\J(cm \cdot I_{\blt})\supseteq \J(c\cdot I_m)=\I(c\log|I_m|)_o\supseteq \I\big(c(m+B)\Phi\big)_o,\]
    where $B=\max_{1\le j\le r} b_j$. If $cm=\lct^{\q}(I_{\blt})$, then
    \[\q\not\subseteq \J(cm\cdot I_{\blt}) \implies \q\not\subseteq \I\big(c(m+B)\Phi\big)_o,\]
    i.e., $\frac{m+B}{m}\cdot\lct^{\q}(I_{\blt})=c(m+B)\ge c_o^{\q}(\Phi)_o$. Hence, we get $\lct^{\q}(I_{\blt})\ge c_o^{\q}(\Phi)_o$ by taking $m\to \infty$.
    
    Take $L\in\bbZ_+$ sufficiently large, and let $d_1,\ldots,d_r\in\bbZ_+$ such that
    \[\frac{d_j}{L}\le \frac{1}{b_j}< \frac{d_j}{L}+\frac{1}{L}, \quad j=1,\ldots r.\]
    Then by \Cref{lem-Skoda.multiplier.ideal}, for every $c>0$, we have
    \[\J(cL \cdot I_{\blt})\subseteq\J(c\cdot \b_L)=\I(c\log|\b_L|)_o\eqqcolon \I(c\phi_L)_o,\]
    where $\b_L=\sum_{j=1}^r \a_j^{d_j}$ and $\phi_L\coloneqq \log|\b_L|$. Note that by the choice of $d_j$,
    \begin{equation*}
        \begin{aligned}
            c\phi_L=\max_{1\le j\le r} c d_j\log|\a_j|+O(1)&\le \max_{1\le j\le r} c\Big(\frac{1}{b_j}-\frac{1}{L}\Big)L\log|\a_j|+O(1)\\
            &\le c(L-B)\Phi+O(1).
        \end{aligned}
    \end{equation*}
    Thus, $\J(cL \cdot I_{\blt})\subseteq \I\big(c(L-B)\Phi\big)_o$ for all $c>0$. Hence, $c_o^{\q}(\Phi)\ge c(L-B)$ for $cL=\lct^{\q}(I_{\blt})$. Consequently, we get $c_o^{\q}(\Phi)\ge \lct^{\q}(I_{\blt})$ by letting $L\to \infty$.
    
    The proof is complete.
\end{proof}

Now we show \Cref{thm-finite.interp} can imply a result in the complex analytic case which was demonstrated in \cite{BGMY25}.

\begin{corollary}[{\cite[Theorem 1.2]{BGMY25}}]
    Let $R=\O_{\bbC^n,o}$, let $\a_1,\ldots,\a_r$ be proper ideals of $R$, and let $b_1,\ldots,b_r\in\bbR_{+}$. Then there exists a valuation on $R$ such that $v(\a_j)=b_j$ for $j=1,\ldots,r$ if and only if $\sigma(\log|\a_1\cdots\a_r|,\Phi)=\sum_{j=1}^r b_j$, where $\Phi=\log\big(|\a_1|^{1/b_1}+\cdots+|\a_r|^{1/b_r}\big)$.
\end{corollary}

\begin{proof}
    Using \Cref{lem-analytic.case.rho-1.equal.sigma}, we get $\sigma(\log|\a_1\cdots\a_r|,\Phi)=\vb(\a_1\cdots\a_r;I_{\blt})$, so this corollary can be induced by \Cref{thm-finite.interp} directly.
\end{proof}

Moreover, \Cref{thm-infinite.interp} gives the following characterization of the existence of infinite valuative interpolation for the complex analytic case. 

\begin{corollary}
    Let $R=\O_{\bbC^n,o}$, let $(\a_j)_{j\ge 1}$ be a sequence of proper ideals of $R$, and let $(b_j)_{j\ge 1}$ be a sequence of positive real numbers. Assume $\bigcap_{j\ge 1}V(\a_j)=\{o\}$ near $o$. Then there exists a valuation $v$ on $R$ centered at $o$ with $A(v)<\infty$ satisfying $v(\a_j)=b_j$ for all $j\ge 1$ if and only if
    \[\sigma(\log|\a_1\cdots\a_r|,\Phi_r)=\sum_{j=1}^r b_j, \quad \forall\, r\ge 1,\]
    and
    \[\sup_{r\ge 1}\sup_{k\in\bbN}\Big(c_o^{\a_1^k\cdots\a_r^k}(\Phi_r)-k\sum_{j=1}^r b_j\Big)<\infty,\]
    where $\Phi_r=\log\big(|\a_1|^{1/b_1}+\cdots+|\a_r|^{1/b_r}\big)$.
\end{corollary}

\begin{proof}
    Using \Cref{lem-analytic.case.rho-1.equal.sigma} and \Cref{lem-monomial.case}, we get for every $k\in\bbN$ and $r\ge 1$,
    \begin{equation*}
            \sigma(\log|\a_1\cdots\a_r|,\Phi_r)=\vb(\a_1\cdots\a_r;I^{(r)}_{\blt}) \quad  \text{and} \quad c_o^{\a_1^k\cdots\a_r^k}(\Phi_r)=\lct(\a_1^k\cdots\a_r^k; I_{\blt}^{(r)}),
    \end{equation*}
    where $I^{(r)}_{\blt}\coloneqq \Big(\frac{1}{b_1}\cdot \a_{1}\Big)_{\blt}+\cdots+\Big(\frac{1}{b_r}\cdot \a_{r}\Big)_{\blt}$. Hence, this corollary can be induced by \Cref{thm-infinite.interp} directly.
\end{proof}

\end{document}